\documentclass[11pt,reqno]{amsart}

\usepackage{amsmath, amssymb, amsfonts, amsbsy, amsthm, latexsym, epsfig}
\usepackage{epstopdf }
\usepackage{graphicx}

\usepackage{
color}
\usepackage{lmodern}
\theoremstyle{plain}
\newtheorem{theorem}{Theorem} [section]
\newtheorem{lemma}[theorem]{Lemma}

\newtheorem{corollary}[theorem]{Corollary}

\theoremstyle{definition}

\newtheorem*{teo*}{Theorem}

\theoremstyle{definition}
\newtheorem{definition}[theorem]{Definition}

\newtheorem{remark}[theorem]{Remark}
\newtheorem{remarks}[theorem]{Remarks}

\newtheorem{problem}[theorem]{Problem}

\def\C{{\mathbb{C}}}
\def\N{{\mathbb{N}}}
\def\R{{\mathbb{R}}}
\def\Z{{\mathbb{Z}}}

\def\DO{{\mathcal{A}}}
\def\D{\mathbb{D}}

\def\Tc{{\mathcal{T}}}

\def\X{{\mathcal{F}}}

\def\li{l_i}

\newcommand{\CHI}{\hbox{\raise .4ex \hbox{$\chi$}}}

\newcommand{\perfect}{perfect}

\def\subset{\subseteq}

\newcommand{\hsd}{H^2(\D)}

\def\la{\lambda}

\begin{document}

\title{ Dynamical sampling }
\author[A. Aldroubi, C. Cabrelli, U. Molter, S. Tang]{A. Aldroubi, C. Cabrelli, U. Molter and S. Tang}

\address{\textrm{(A. Aldroubi)}
Department of Mathematics,
Vanderbilt University,
Nashville, Tennessee 37240-0001 USA}
\email{aldroubi@math.vanderbilt.edu}

\address{\textrm{(C. Cabrelli)}
Departamento de Matem\'atica,
Facultad de Ciencias Exac\-tas y Naturales,
Universidad de Buenos Aires, Ciudad Universitaria, Pabell\'on I,
1428 Buenos Aires, Argentina and
CONICET, Consejo Nacional de Investigaciones
Cient\'ificas y T\'ecnicas, Argentina}
\email{cabrelli@dm.uba.ar}

\address{\textrm{(U. Molter)}
Departamento de Matem\'atica,
Facultad de Ciencias Exac\-tas y Naturales,
Universidad de Buenos Aires, Ciudad Universitaria, Pabell\'on I,
1428 Buenos Aires, Argentina and
CONICET, Consejo Nacional de Investigaciones
Cient\'ificas y T\'ecnicas, Argentina}
\email{umolter@dm.uba.ar}

\address{\textrm{(S. Tang)}
Department of Mathematics,
Vanderbilt University,
Nashville, Tennessee 37240-0001 USA}
\email{stang@math.vanderbilt.edu}

\thanks{The research of
A.~Aldroubi and S.~Tang is supported in part by NSF Grant DMS- 1322099.
C.~Cabrelli and U.~Molter are partially supported by
Grants  PICT 2011-0436 (ANPCyT), PIP 2008-398 (CONICET)
and UBACyT  20020100100502 and 20020100100638 (UBA).
}


\date{\today}

\begin{abstract}
 Let $Y=\{f(i), Af(i), \dots, A^{\li}f(i): i \in \Omega\}$, where $A$ is a bounded operator on $\ell^2(I)$. The problem under consideration is to find necessary and sufficient conditions on $A, \Omega, \{l_i:i\in\Omega\}$ in order to recover any $ f \in \ell^2(I)$ from the measurements $Y$. This is the so called dynamical sampling problem in which we seek to recover a function $f$ by combining coarse samples of $f$ and its futures states $A^lf$.  We completely solve this problem in finite dimensional spaces, and for a large class of self adjoint operators in infinite dimensional spaces. In the latter case, the M\"untz-Sz\'asz Theorem  combined with the Kadison-Singer/Feichtinger Theorem allows us to show that $Y$ can never be a Riesz  basis when $\Omega$ is finite. We can also show that, when $\Omega$ is finite,  $Y=\{f(i), Af(i), \dots, A^{\li}f(i): i \in \Omega\}$ is not a frame except for some very special cases. The existence of these special cases is derived from  Carleson's Theorem for interpolating sequences in the Hardy space $\hsd$.

\end{abstract}

\maketitle

\section{Introduction}

Dynamical sampling refers to the process  that results from sampling
an evolving signal $f$ at various times and asks the question: when do coarse samplings
taken at varying times contain the same information as a finer sampling taken at the
earliest time? In other words, under what conditions on an evolving system, can time
samples be traded for spatial samples? Because dynamical sampling uses samples from varying time levels for a single reconstruction,
it departs from classical sampling theory in which a signal $f$ does not evolve in time and is to be reconstructed from its samples at a single time $t=0$, see  \cite[and references therein]{AH12, ABK08, AG01, BG12, BF01, CRT06, GKKY12, G04, HL12, HNS09, J11, LM94,NS10,S01,Su06}.

 The general dynamical sampling problem can be stated as follows: Let $f$ be a function in a  separable Hilbert space $\mathcal H$, e.g., $\C^d$ or $\ell^2(\N)$, and assume that $f$ evolves through an evolution operator $A: \mathcal H \to \mathcal H$ so that the function at time $n$ has evolved to become $f^{(n)}=A^nf$. We identify $\mathcal H$ with $\ell^2(I)$ where $I=\{1,\dots,d\}$  in the finite dimensional  case,  $I =\N$ in the infinite dimensional case. We denote by  $\{e_i\}_{i\in I}$ the standard basis of $\ell^2(I)$.
 
 The time-space sample at time $t\in \N$ and location $p\in I$, is the value $A^tf(p)$. In this way we associate to each pair $(p,t)\in I\times \N$
 a sample value.  
 
  The general dynamical sampling problem can then be described as: Under what conditions on the operator $A$, 
  and a set $S\subset I\times \N$, can  every $f$ in the Hilbert space $H$ be recovered in a stable way from the samples in $S$.
   
 At  time $t=n$, we sample $f$ at the locations $\Omega_n \subset I$ resulting in the measurements $\{f^{(n)}(i): i \in \Omega_n\}$.
 Here  $f^{(n)}(i)=< A^nf,e_i>.$

The measurements $\{f^{(0)}(i): i \in \Omega_0\}$ that we have from our original signal $f=f^{(0)}$ will contain in general insufficient information to 
recover $f$. In other words, $f$ is undersampled. So we will need some extra information from the iterations of $f$ by the operator $A$: $\{f^{(n)}(i)=A^nf(i): i \in \Omega_n\}$. Again, for each $n$, the measurements $\{f^{(n)}(i): i \in \Omega_n\}$ that we have by sampling our signals $A^nf$ at $\Omega_n$ are insufficient to recover $A^nf$ in general. 

Several questions arise.
Will the combined measurements $\{f^{(n)}(i): i \in \Omega_n\}$  contain in general all the information needed to recover $f$ (and hence $A^nf$)? How many iterations $L$ will we need  (i.e., $n=1,\dots,L$) 
to recover the original signal?  What are the right ``spatial'' sampling sets $\Omega_n$ we need to
 choose in order to recover $f$? In what way all these questions depend on the operator $A$?

The goal of this paper is to answer these questions and understand completely this problem that we can formulate as:

Let $A$ be the evolution operator acting in $\ell^2(I)$, $\Omega \subset I$ a fixed set of locations, and  $\{\li:i\in \Omega\}$ where $\li$ is a positive integer or $+\infty.$

\begin {problem} \label {GenDysamp} Find conditions on  $A, \Omega$ and $\{\li:i\in \Omega\}$ such that any vector $f\in \ell^2(I)$ can be recovered from the  samples \;
$Y=\{f(i), Af(i), \dots, A^{\li}f(i): i \in \Omega\}$ in a stable way.
\end {problem}
 Note that, in Problem \ref{GenDysamp}, we allow $l_i$ to be finite or infinite. Note also that, Problem \ref {GenDysamp} is not the most general problem since the way it is stated implies that $\Omega=\Omega_0$ and $\Omega_n=\{i \in \Omega_0: l_i\ge n\}$. Thus, an underlying assumption is that $\Omega_{n+1}\subset \Omega_n$ for all $n\ge 0$. 
 For each $i\in\Omega$, let $S_i$ be the operator from $\mathcal H=\ell^2(I)$ to $\mathcal H_i=\ell^2(\{0,\dots,\li\})$, defined by $S_i f= (A^jf(i))_{j=0,\dots,\li}$ and define $S$ to be the operator $S=S_0 \oplus S_1\oplus \dots$ 

Then $f$ can be recovered from $Y=\{f(i), Af(i), \dots, A^{\li}f(i): i \in \Omega\}$ in a stable way if and only if there exist constants $c_1,c_2>0$ such that 
\begin {equation}
\label {eqnorm}
c_1\|f\|^2_2\le\|\mathcal S f\|^2_2=\sum\limits_{i\in \Omega} \|S_if\|_2^2\le c_2\|f\|^2_2.
\end{equation}
Using the standard basis  $\{e_i\}$ for $\ell^2(I)$, we obtain from \eqref {eqnorm} that
\[c_1\|f\|^2_2\le\sum\limits_{i \in \Omega} \sum\limits_{j=0}^{\li}|\langle f,A^{\ast j} e_i\rangle|^2\le c_2\|f\|^2_2.\]
Thus we get
\begin {lemma} \label {Eqv1}Every $f \in \ell^2(I)$  can be recovered from the measurements set $Y=\{f(i), Af(i), \dots, A^{\li}f: i \in \Omega\}$ in a stable way if and only if  the set of vectors $\{A^{\ast j}e_i: \; i\in \Omega,  \, j=0, \dots, \li\}$ is a frame for $\ell^2(I)$.
\end {lemma}

 \subsection{Connections to other fields }
 The dynamical sampling problem has similarities to other areas of mathematics. For example, in wavelet theory \cite {BJ02,CM11,D92,HW96,M98,OS04,SN96}, a high-pass convolution operator $H$ and a low-pass convolution operator $L$ are applied to the
 function $f$. The goal is to design operators $H$ and $L$ so that reconstruction of $f$ from samples
of $Hf$ and $Lf$ is feasible. In dynamical sampling there is only one operator $A$, and it is
applied iteratively to the function $f$. Furthermore, the operator $A$ may be high-pass, low-pass,
or neither and is given in the problem formulation, not designed.
 
In inverse problems (see \cite {N11} and the references therein), a single operator $B$, that often represents a physical process,
is to be inverted. The goal is to recover a function $f$ from the observation  $Bf$. If
$B$ is not bounded below, the problem is considered an ill-posed inverse problem. 
 Dynamical sampling is different because $A^nf $ is not necessarily known for any $n$; instead $f$ is to be
recovered from partial knowledge of $A^nf$ for many values of $n$. In fact, the dynamical
sampling problem can be phrased as an inverse problem when the operator $B$ is the
operation of applying the operators $A, A^2,\dots, A^L$ and then subsampling each of these
signals accordingly on some sets $\Omega_n$ for times $t=n$.
 
 The methods that we develop for studying the dynamical sampling problem are related to methods in spectral theory, operator algebras, and frame theory \cite {ABK08, CCL12,CKL08, Con96, FS10, G04,GL04,HL08, S08ACM}. 
 For example, the proof of our Theorems \ref {Noframe0} and \ref {NotFrame2},  below, use the newly proved \cite {MSS13} Kadison-Singer/Feichtinger conjecture \cite {CTC06, CCLV05}. Another example is the existence of cyclic vectors that form frames, which is related to Carleson's Theorem for interpolating sequences in the Hardy space $\hsd$ (c.f., Theorem \ref {OnePointFrame}).

Application to Wireless Sensor Networks (WSN) is a natural setting for dynamical sampling. In WSN,
large amounts of physical sensors are distributed to gather information about a field
to be monitored, such as temperature, pressure, or pollution. WSN are used in many
industries, including the health, military, and environmental industries (c.f., \cite{HRLV10,LDV11,RCLV11,LV09,RMG12,RM10} and the reference therein). The  goal  
 is to exploit the evolutionary structure and the placement 
of sensors to reconstruct an unknown  field. The idea is simple. If it is not possible to place sampling devices at the desired locations, then we may be able to recover
the desired information by placing the sensors elsewhere and use the evolution process to recover the signals at the relevant locations. In addition, if the cost of a
sensor is expensive relative to the cost of activating the sensor, then, we may be able to recover
the same information with fewer sensors, each being activated more frequently. In this way, reconstruction of a signal becomes cheaper. In other words we perform a time-space trade-off. 

 \subsection{ Contribution and organization} In section \ref {FDC} we present the results  for the  finite dimensional case. Specifically, Subsection \ref {DT} concerns the special case of diagonalizable operators acting on vectors in $\C^d$. This case is treated first in order  to give some intuition about the general theory. For example, Theorem \ref {TAD0} explains the reconstruction properties for the examples below:
 Consider the following two matrices acting on $\C^5$.
    \begin {equation*}\label {A1}
P=
  \left(\begin{array} {rrrrr}
 9/2 &   1/2 &  -7 &  5  & -3\\
   15/2  &  3/2  & -11  &  5 &  -7\\
    5   &      0  & -7  &  5  & -5\\
    4    &     0   & -4  &  3  & -4\\
    1/2  & 1/2  & -1   &      0  &  1
     \end{array}\right) \quad
     Q =  \left(\begin{array} {rrrrr} 
        3/2&   -1/2 &    2 &         0  &  1\\
    1/2 &    5/2 &         0    &      0  &  -1 \\
         0    &      0   &  3 &         0     &     0 \\
    1  &         0   &  -1  &    3  &   -1 \\
   -1/2 &   -1/2 &   1 &         0   &  3 
 \end{array}\right) .
 \end {equation*}
 
 For the matrix $P$, Theorem \ref {TAD0} shows that any $f \in \C^5$ can be recovered from the data
sampled at the single ``spacial'' point $i = 2$, i.e., from
\[Y = \{f(2),Pf(2),P^2
f(2),P^3
f(2),P^4f(2)\}.
 \]
However, if $i = 3$, i.e., $Y = \{f(3),Pf(3),P^2f(3),P^3f(3),P^4f(3)\}$ 
 the information is not sufficient to determine $f$. In fact if we do not sample at $i = 1$, or
$i = 2$, the only way to recover any $f \in \C^5$ is to sample at all the remaining ``spacial'' points $i = 3, 4, 5$. For example, $Y = \{f(i),Pf(i) : i = 3, 4, 5\}$ is
enough data to recover $f$, but $Y = \{f(i),Pf(i), . . . ,P^Lf(i) : i = 3, 4\}$, is
not enough information no matter how large $L $ is.

 For the matrix $Q$, Theorem \ref {TAD0} implies that it is not possible to reconstruct
$f \in \C^5$ if the number of sampling points is less than $3$. However, we
can reconstruct any $f \in \C^5$ from the data
 \begin {align*}
Y =&\{f(1),Qf(1),Q^2f(1),Q^3f(1),Q^4f(1),\\
&f(2),Qf(2),Q^2f(2),Q^3f(2),Q^4f(2),\\
&f(4),Qf(4)\}.
\end{align*}
Yet, it is not possible to recover $f$ from the set $Y = \{Q^lf(i) : i = 1, 2, 3, l =
0, \dots ,L\}$ for any $L$. Theorem \ref {TAD0} gives all the sets $\Omega$ such that any $f \in \C^5$ can be recovered from $Y = \{A^lf(i) :i \in \Omega, l = 0, . . . \li\}$.  

In subsection \ref {GLT} Problem \ref {GenDysamp} is solved for the general case in $\C^d$, and
Corollary \ref {tadcul-cor} elucidates the example below: Consider 
    \begin {equation*}\label {A3}
R=
 \left(\begin{array} {rrrrr}  0  & -1 &    4 &  -1 &    2\\
    2&     1 &   -2 &    1 &   -2\\
   -1/2 &   -1/2 &    3 &        0  &  1\\
    1/2 &   -1/2 &         0  &  2 &         0\\
   -1/2 &   -1/2 &    2 &   -1 &    2
    \end{array}\right).
 \end {equation*}

  Then, Corollary \ref {tadcul-cor} shows that $\Omega$ must contain at least two ``spacial''
sampling points for the recovery of functions from their time-space samples to 
be feasible. For example, if $\Omega = \{1, 3\}$, then $Y = \{R^lf(i) : i \in \Omega, l =
0, \dots ,L\}$ is enough recover $f\in \C^5$. However, if $\Omega$ is changed to $\Omega=\{1, 2\}$, then
$Y = \{R^lf(i) : i \in \Omega, l = 0, \dots ,L\}$ does not provide enough information.
  
  The dynamical sampling problem in infinite dimensional separable Hilbert
spaces is studied in Section \ref {DSID}. For this case, we restrict ourselves to certain
classes of self adjoint operators in $\ell^2(\N)$. In light of Lemma \ref {Eqv1}, in Subsection
\ref {completeness}, we characterize the sets $\Omega \subset \N$ such that $\X_\Omega=\{A^je_i : i \in \Omega, j = 0,\dots, \li\}$ is complete in $\ell^2(\N)$ (Theorem \ref  {complete}). However, using the newly proved [28]
Kadison-Singer/Feichtinger conjecture [11, 9], we also show that if $\Omega$ is a
finite set, then $\{A^je_i : i \in \Omega, j = 0,\dots , \li\}$ is never a basis (see Theorem
\ref {NonBasisFiniteOmega}). It turns out that the obstruction to being a basis is redundancy.  This fact  is proved
using the beautiful M\"untz-Sz\'asz Theorem  \ref  {MS}  below. 

Although $\X_\Omega=\{A^je_i : i \in \Omega, j = 0,\dots, \li\}$ cannot be a basis, it should be possible that $\X_\Omega$ is a frame for sets $\Omega\subset \N$ with
finite cardinality. It turns out however, that except for special cases, if $
  \Omega$ is a finite
set, then $\X_\Omega$ is not a frame for $\ell^2(\N)$. 

If $\Omega$ consists of a single vector, we are able to characterize completely when  $\X_\Omega$ is a frame for $\ell^2(\N)$ (Theorem
\ref {OnePointFrame}), by relating our problem to
 a  theorem by Carleson
on interpolating sequences in the Hardy spaces $\hsd$.

\section{Finite dimensional case}
\label {FDC}
In this section we will address the finite dimensional case. That is, our evolution operator is a matrix $A$ acting on the space $\C^d$
and $I= \{1,\dots,d\}$.
Thus, given $A$, our goal is to find necessary and sufficient conditions on the set of indices $\Omega \subset I$ and the numbers $ \{\li\}_{i\in \Omega}$ such that
every vector $f\in \C^d$ can be recovered from the samples $\{A^j f(i): i\in \Omega, \;j=0,\dots,\li\}$ or equivalently (using Lemma \ref{Eqv1}), the set of vectors

\begin{equation}\label{frame}
\{A^{\ast j} e_i: i\in \Omega,\; j=0,\dots,\li\} \text{ is a frame of } \C^d.
\end{equation}
(Note that this implies that we need at least $d$ space-time samples to be able to recover the vector $f$).

The problem can be further reduced as follows:
Let $B$ be any invertible matrix with complex coefficients, and let $Q$ be the matrix $Q=BA^*B^{-1}$, so that $A^* = B^{-1}QB$.
Let  $b_i$ denote the ith column of $B$.
Since a frame is transformed to a frame by invertible linear operators, condition (\ref{frame})  is equivalent to $\{Q^jb_i: i \in \Omega, \; j=0,\dots, \li\}$ being a frame of $\C^d$.

This allows us to replace the general matrix $A^*$ by a possibly simpler matrix and we have:

 \begin {lemma} \label {Eqv2-diag} Every $f \in \C^d$  can be recovered from the measurement set $Y=\{A^jf(i): \; i \in \Omega, j=0,\dots, \li\}$ if and only if  the set of vectors $\{Q^jb_i: i \in \Omega, \; j=0,\dots, \li\}$ is a frame for $\C^d$.
\end {lemma}

We  begin with the simpler case when $A^*$ is a diagonalizable matrix.

\subsection{Diagonalizable Transformations} \label {DT}
Let $A\in \C^{d\times d}$ be a matrix that can be written as $A^\ast=B^{-1}DB$ where $D$ is a diagonal matrix of the form

\begin {equation}\label {Adiag}
D=
 \begin{pmatrix}
 
  \lambda_1 I_1 & 0 & \cdots & 0 \\
  0 & \lambda_2I_2 & \cdots & 0 \\
  \vdots  & \vdots  & \ddots & \vdots  \\
  0 & 0 & \cdots & \lambda_n I_n
 \end{pmatrix}.
 \end {equation}
In \eqref {Adiag}, $I_k$ is an $h_k\times h_k$ identity matrix, and $B\in\C^{d\times d}$ is an invertible matrix. Thus $A^\ast$ is a diagonalizable matrix with distinct eigenvalues $\{\lambda_1,\dots,\lambda_n\}$.

Using Lemma \ref {Eqv2-diag} and $Q=D$,  Problem \ref {GenDysamp} becomes the problem of finding necessary and sufficient conditions on vectors $b_i$ and numbers $\li$, and the set $\Omega \subset \{1,\dots,m\}$ such that the set of vectors $\{D^jb_i: i \in \Omega, \; j=0,\dots, \li\}$ is a frame for $\C^d$. 
Recall that the {\em $Q$-annihilator $q^Q_b$ of a vector $b$} is the monic polynomial of smallest degree, such that $q_b^Q(Q)b \equiv 0$. 
Let  $P_j$ denote 
the orthogonal projection in $\C^d$ onto the eigenspace of $D$ associated to the eigenvalue $\lambda_j$. Then we have:

\begin {theorem}\label {TAD0}

Let $\Omega \subset \{1, \dots, d\}$ and let $\{b_i: i \in \Omega \}$ be vectors in $\C^d$. Let $r_i$ be the degree of the $D$-annihilator of $b_i$ and let $l_i=r_i-1$. Then  $\{D^{j}b_i: \; i\in \Omega,  \, j=0, \dots, \li\}$ is a frame of $\C^d$ if and only if $\{P_j(b_i):i \in \Omega\}$ form a frame of $P_j(\C^d)$, $j=1,\dots ,n$.
\end{theorem}

As a corollary, using Lemma \ref {Eqv2-diag} we get  
\begin {theorem}\label {TAD1}

Let $A^\ast=B^{-1}DB$, and  let $\{b_i: i\in \Omega  \}$ be the column vectors of $B$ whose indices belong to $\Omega$. Let $r_i$ be the degree of the $D$-annihilator of $b_i$ and let $l_i=r_i-1$. Then  $\{A^{\ast j}e_i: \; i\in \Omega,  \, j=0, \dots, \li\}$ is a frame of $\C^d$ if and only if $\{P_j(b_i):i \in \Omega\}$ form a frame of $P_j(\C^d)$, $j= 1 ,\dots ,n$. 

Equivalently, any vector $f \in \C^d$ can be recovered from  the  samples 
$Y=\{f(i), Af(i), \dots, A^{\li}f(i): i \in \Omega\}$ if and only if $\{P_j(b_i):i \in \Omega\}$ form a frame of $P_j(\C^d)$, $j=1,\dots ,n$.
\end{theorem}

Note that, in the previous  Theorem, the number of time-samples $l_i$ depends on the sampling point $i$. If instead the number of time-samples $L$ is the same for all $i \in \Omega$, (note that $L \geq max \{l_i:i \in \Omega\}$ is an obvious choice, but depending on the vectors $b_i$ it may be  possible to choose  $L\leq min \{l_i:i \in \Omega\}$), then we have the following Theorems (see Figures \ref {fig1})
\begin {theorem}\label{TAL0}
Let $\Omega \subset \{1, \dots, d\}$ and $\{b_i: i \in \Omega \}$ be a set of vectors in $\C^d$ such that $\{P_j(b_i): i\in \Omega\}$ form a frame of $P_j(\C^d)$, $j=1,\dots ,n$. Let $L$ be any fixed integer, then $E=\bigcup\limits_{\{i \in \Omega: b_i\neq 0\}} \{b_i,Db_i,\dots, D^{L}b_i\}$ is a frame of $\C^d$ if and only if $\{D^{L+1}b_i,: i \in \Omega\} \subset span (E)$. 
\end{theorem}

As a corollary, for our original problem \ref  {GenDysamp} we get 
\begin {theorem}\label{TAL}

Let  $A^\ast=B^{-1}DB$,  $L$ be any fixed integer, and let $\{b_i: i\in \Omega \}$ be a set of vectors in $\C^d$ such that $\{P_j(b_i): i \in \Omega\}$ form a frame of $P_j(\C^d)$, $j=1,\dots ,n$. Then $\{A^{\ast j}e_i: \; i\in \Omega,  \, j=0, \dots, L\}$ is a frame of $\C^d$ if and only if $\{D^{L+1}b_i: i \in \Omega\} \subset  span \big(\{D^jb_i: i\in \Omega\; , j=0, \dots, L\}\big)$. 

Equivalently any $f\in \C^d$ can be recovered from the  samples $$Y=\{f(i), Af(i), A^2f(i), \dots,A^{L}f(i):  i \in \Omega\},$$ if and only if $\{D^{L+1}b_i: i \in \Omega\} \subset  span \big(\{D^jb_i: i\in \Omega\; , j=0, \dots, L\}\big)$.
\end{theorem}

\begin{figure}[h]
	\centering
		\includegraphics[scale=0.36]{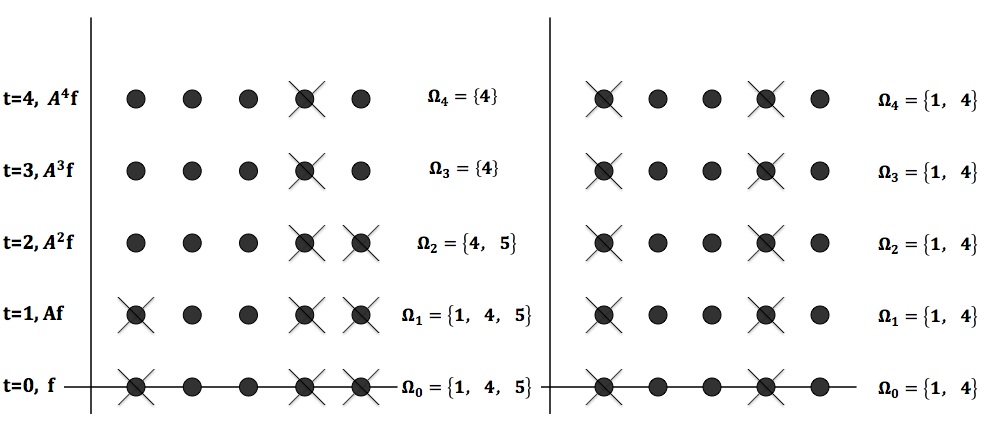}
	\caption{Illustration of a time-space sampling pattern. Crosses correspond to  time-space sampling points. Left panel: $\Omega=\Omega_0=\{1,4,5\}$. $l_1=1, l_4=4,l_5=3$. Right panel:  $\Omega=\Omega_0=\{1,4\}$. $L=4$.}
	\label{fig1}
\end{figure}

Examples where $L < d$, while $l_i=d$ for all $ i \in \Omega$ can be found in \cite {ADK13}.

Theorems \ref {TAD1} and \ref{TAL} will be consequences of our general results but we state them here to help the comprehension of the general results below.

\subsection {General linear transformations}
\label {GLT}

For a general matrix  we will need to use the reduction to its Jordan form. To state our results in this case, we need to introduce some notations
and describe the general Jordan form of a matrix with complex entries.   (For these and other results about matrix or linear transformation decompositions see for example \cite{HK71}.)

A matrix $J$ is in Jordan form if 
\begin {equation}\label {AJord}
J=
 \begin{pmatrix}
 
J_1 & 0 & \cdots & 0 \\
  0 & J_2 & \cdots & 0 \\
  \vdots  & \vdots  & \ddots & \vdots  \\
  0 & 0 & \cdots & J_n
 \end{pmatrix}.
 \end {equation}
In \eqref {AJord}, for $s = 1, \dots, n$, $J_s = \lambda_s I_s + N_s$ where $I_s$ is an $h_s\times h_s$ identity matrix, and $N_s$ is a $h_s \times h_s$ nilpotent block-matrix of the form:
\begin {equation}\label {Nilp}
N_s =
 \begin{pmatrix}
   N_{s1} & 0 & \cdots & 0 \\
  0 & N_{s2} & \cdots & 0 \\
  \vdots  & \vdots  & \ddots & \vdots  \\
  0 & 0 & \cdots & N_{s \gamma_s}
 \end{pmatrix}  \end {equation} 
where each $N_{si}$ is a $t^{s}_i \times t^{s}_i$ cyclic nilpotent matrix,  
\begin {equation}\label {cyc}
N_{si} \in \C^{t^s_i\times t^s_i}, \quad N_{si} = 
\begin{pmatrix}
0 & 0 & 0 & \cdots & 0 & 0\\
1 & 0 & 0 & \cdots & 0 & 0\\
0 & 1 & 0 & \cdots & 0 & 0\\
\vdots  & \vdots & \ddots  & \vdots & \vdots  & \vdots\\
0 & 0 & 0 & \cdots & 1 & 0
\end{pmatrix},
\end {equation}
with  $t^{s}_1\geq t^{s}_2 \geq \dots$, and  $t^{s}_1 + t^{s}_2 + \dots + t^s_{\gamma_s} = h_s$. Also $h_1 + \dots + h_n = d$. 
The matrix $J$ has $d$ rows and distinct eigenvalues $\lambda_j, j = 1, \dots, n$. 

Let $k_j^s$ denote the index corresponding to the first row of the block $N_{sj}$ from the matrix $J$, and let    $e_{k_j^s}$ be the corresponding element of the standard basis of $\C^d$. (That is a cyclic vector associated to that block). 
We also define $W_s := \text{span} \{e_{k^s_j}: j=1,\dots, \gamma_s\}$, for $s=1, \dots,n$, and $P_s$ will again denote the orthogonal projection onto $W_s$. Finally, recall that the {\em $J$ annihilator $q^J_b$ of a vector $b$} is the monic polynomial of smallest degree, such that $q_b^J(J)b \equiv 0$. Using the notations and definitions above we can state the following theorem:
\begin {theorem}
\label {tadcul}
Let $J$ be a matrix in Jordan form,  as in  \eqref{AJord}. 
Let $ \Omega \subset \{1, \dots, d\}$ and $\{b_i: i \in \Omega\}$ be a subset of vectors of $\C^d$, $r_i$ be the degree of  the $J$-annihilator of the vector $b_i$ and let $l_i=r_i-1$. 

Then the following propositions are equivalent.
\begin{enumerate}

\item[i)]
The set of vectors $\{J^jb_i: \; i \in \Omega, j=0,\dots, \li\}$ is a frame for $\C^d$.
\item[ii)]
For every $s = 1, \dots, n$, \;$\{P_s(b_i), i \in \Omega\}$ form a frame of $W_s.$
\end{enumerate}
\end{theorem}

Now, for a general matrix $A$, using Lemma \ref{Eqv2-diag} we can state:

\begin {corollary}\label {tadcul-cor}
Let $A$ be a matrix, such that $A^* = B^{-1}JB$, where $J \in \C^{d\times d}$ is the Jordan matrix for $A^*$. Let $\{b_i: i \in \Omega\}$ be a subset of the column vectors of $B$, $r_i$ be the degree of  the $J$-annihilator of the vector $b_i$, and  let $\li=r_i-1$.\\
Then, every $f \in \C^d$  can be recovered from the measurement set $Y=\{(A^jf)(i): \; i \in \Omega, j=0,\dots, \li\}$ of $\C^d$ if and only if $\{P_s(b_i), i \in \Omega\}$ form a frame of $W_s$. 
\end{corollary}
In other words, we will be able to recover $f$ from the measurements $Y$, if and only if the {\em Jordan}-vectors of $A^*$  (i.e. the columns of the matrix $B$ that reduces $A^*$ to its Jordan form) corresponding to $\Omega$ satisfy that their projections on the spaces $W_s$ form a frame.

\begin{remark}
We want to emphasize at this point, that given a matrix in Jordan form there is an obvious choice 
of vectors in order that their iterations give a frame of the space,  (namely, the cyclic vectors $e_{k^s_j}$
corresponding to each block). However, we are dealing here with a much more difficult problem. The vectors $b_i$ are given beforehand, and we need to find  conditions in order to decide if their iterations form a frame.
\end{remark}
The following theorem is just a statement about replacing the {\em optimal} iteration of each vector $b_i$ by any fixed number of iterations. The idea is, that we iterate a fixed number of times $L$   but we do not need to know the  degree $r_i$ of the $J$-annihilator for each $b_i$. Clearly, if $L\ge \max\{r_i-1: i \in \Omega\}$ then we can always recover any $f$ from $Y$. But the number of time iterations $L$  may be smaller than any $r_i-1$,  $i \in \Omega$. In fact, for practical purposes it might be better to iterate, than to try to figure out which is the degree of the annihilator for $b_i$.

\begin {theorem}\label {tacul}
Let $J \in \C^{d\times d}$ be a matrix in Jordan form (see \eqref{AJord}).
Let $\Omega \subset \{1, \dots, d\}$, and let $\{b_i: i \in \Omega\}$ be a set of vectors in $\C^d$, such that for each $s = 1, \dots, n$ the projections $\{P_s(b_i): i \in \Omega\}$ onto $W_s$ form a frame of $W_s$. 
Let $L$ be any fixed integer, then $E=\bigcup\limits_{\{i\in \Omega: b_i\neq 0\}} \{b_i,Jb_i,\dots, J^{L}b_i\}$ is a frame of $\C^d$ if and only if $\{J^{L+1}b_i: i \in \Omega\} \subset span (E)$.
\end{theorem}
As a corollary we immediately get the solution to Problem \ref {GenDysamp} in finite dimensions.

\begin {corollary}
Let $\Omega \subset I$, $A^\ast=B^{-1}JB$,  and $L$ be any fixed integer. 
Assume that  $\{P_s(b_i): i \in \Omega\}$ form a frame of $W_s$ and set  $E=\{J^sb_i: i \in \Omega, s=0, \dots, L,\}$.
Then any $f\in \C^d$ can be recovered from the  samples 
$Y=\{f(i), Af(i), A^2f(i), \dots,A^{L}f(i):  i \in \Omega\}$, if and only if $\{J^{L+1}b_i: i \in \Omega\} \subset \text{span} ( E\}).$
\end {corollary}


\subsection{Proofs}
\ 

In order to introduce some needed notations, we first recall the standard decomposition of a linear transformation acting on a
finite dimensional vector space that produces a basis for the Jordan form.

Let $V$ be a finite dimensional vector space of dimension $d$ over $\C$  and let $T:V \longrightarrow V$ be a linear transformation.
The characteristic polynomial of $T$ factorizes as $\chi_T(x) = (x-\lambda_1)^{h_1} \dots (x-\lambda_n)^{h_n}$ where $h_i \geq 1$ and  $\lambda_1,\dots,\lambda_n$  are distinct elements of $\C$.
The minimal polynomial of $T$ will be then $m_T(x) = (x-\lambda_1)^{r_1}\dots (x-\lambda_n)^{r_n}$ with $1 \leq r_i \leq h_i$ for $i=1, \dots, n$. By the primary decomposition theorem, the subspaces $V_s = \text{Ker}(T-\lambda_s I)^{r_s}, \; s=1, \dots, n$  are invariant under $T$ (i.e. $T(V_s) \subset V_s$) and we have also that $V = V_1\oplus \dots \oplus V_n$.

Let  $T_s$ be the restriction of $T$ to $V_s$.  Then, the minimal polynomial of $T_s$ is $(x-\lambda_s)^{r_s}$,
and $T_s = N_s+\lambda_sI_s$, where $N_s$ is nilpotent of order $r_s$ and  $I_s$ is the identity operator on $V_s$.
Now for each $s$ we apply the cyclic decomposition to $N_s$ and the space $V_s$ to obtain:
$$
V_s = V_{s1}\oplus \dots \oplus V_{s\gamma_s} 
$$
where each $V_{sj}$ is invariant under $N_s$, and the restriction operator $N_{sj}$ of  $N_s$ to $V_{sj} $ is a {\em cyclic nilpotent } operator
on $V_{sj} $. 

Finally, let us fix for each $j$ a cyclic vector $w_{sj}\in V_{sj} $  and define the subspace $W_s = \text{span}\{w_{s1}\dots w_{s\gamma_s}\}$,
$W=W_1\oplus \dots \oplus W_n$ and let $P_{W_s}$ be the projection onto $W_s$, with $I_W = P_{W_1} +\dots+ P_{W_n}.$

With this notation we can state the main theorem of this section:

\begin{theorem}\label{Jordan-N}
Let $\{b_i: i \in \Omega\}$ be a set of vectors in $V$. If the set $\{P_{W_s}b_i : i \in \Omega\}$ is complete in $W_s$ for each $ s= 1,\dots, n$, then the set $\{b_i, Tb_i,\dots,T^{\li}b_i: i\in \Omega\} $ is a frame of $V$, where $r_i$ is the degree of the $T$-annihilator of $b_i$ and $\li=r_i-1$.
\end{theorem}


To prove Theorem \ref{Jordan-N}, we will first concentrate on the case where the transformation $T$ has minimal polynomial consisting of a unique factor, 
i.e. $m_T(x) = (x-\lambda)^r$, so that $T = \lambda I_d + N$, and $N^r = 0$ but $N^{r-1} \not=0$.

\subsection{Case $T = \lambda I_d +N$}

\begin{remark} It is not difficult to see that, in this case, given   some $L \in \N$, $\{T^jb_i: \;  i\in \Omega, j=0,\dots, L \}$ is a frame for $V$
if and only if $\{N^jb_i: \; i\in \Omega, j=0,\dots, L  \}$ is a frame for $V$.
In addition, since $N^r b_i = 0$ we need only to iterate to $r-1$. In fact, we only need to iterate each $b_i$ to $l_i=r_i-1$ where $r_i$ is  the degree of the $N$ annihilator of $b_i$.
\end{remark}

\begin{definition} A matrix $A \in \C^{d\times d}$ is {\em \perfect} if $a_{ii}\not=0, i=1, \dots, d$ and $\text{det}(A_i) \not=0, i=1, \dots, d$ where $A_{s} \in \C^{s \times s}$ is the submatrix of $A$, $A_{s} = \{a_{i,j}\}_{i,j = 1, \dots, s}$.
\end{definition}

We need the following lemma that is straightforward to prove.
\begin{lemma} \label{perfect} Let $A \in \C^{d\times d}$ be an invertible matrix. Then there exists a \perfect\ matrix $B \in \C^{d\times d}$ that consists of  row (or column) permutations of $A$.
\end{lemma}

If $N$ is nilpotent of order $r$, then there exist $\gamma \in \N$ and invariant subspaces $V_i \subset V$, $i=1, \dots, \gamma$ such that
\begin{equation*}
V = V_1 \oplus \dots \oplus V_\gamma , \quad \text{dim}(V_j) = t_j, \  t_j \geq t_{j+1}, \ j = 1, \dots, \gamma-1, \end{equation*}
and $N = N_1 + \dots +N_\gamma$, 
where $N_j = P_jNP_j$ is a cyclic nilpotent operator in $V_j$, $j = 1, \dots, \gamma$. Here $P_j$ is the projection onto $V_j$. Note that $t_1+\dots+t_\gamma = d$.

For each $j=1,\dots,\gamma$, let $w_j \in V_j$ be a cyclic vector for $N_j$. Note that the set $\{w_1, \dots, w_\gamma\}$ is a linearly independent set.

Let $W = \text{span}\{w_1, \dots, w_\gamma\}$. Then, we can write $V= W\oplus NW \oplus \dots \oplus N^{r-1} W$. Furthermore, the projections $ P_{N^jW}$ satisfy $P^2_{N^jW} = P_{N^jW}$, and $I = \sum_{j=0}^{r-1} P_{N^jW}$.

Finally, note that 
\begin{equation}\label{nilp}
N^sP_{W} = P_{N^sW} N^s.
\end{equation}

With the notation above, we have the following theorem:
\begin{theorem}\label{Jordan-1block} Let $N$ be a nilpotent operator on $V$. Let $B\subset V$ be a finite set of vectors such that $\{P_W(b): b \in B\}$ is complete in $W$. Then
\begin{equation*}
\bigcup _{b \in B} \left\{b, Nb, \dots, N^{l_b }b\right\} \quad \text{is a frame for }\  V,
\end{equation*}
where $l_b=r_b-1$ and $r_b$ is the degree of the $N$-annihilator of $b$.
\end{theorem}
\begin{proof}
In order to prove  Theorem \ref {Jordan-1block}, we will show that there exist  vectors $\{b_1, \dots, b_\gamma\}$ in $B$, where $\gamma = \text{dim}(W)$, such that 
\begin{equation*}
\bigcup _{i=1}^\gamma \left\{b_i, Nb_i, \dots, N^{t_i -1}b_i\right\} \quad \text{is a basis of}\  V.
\end{equation*}
Recall that $t_i$ are the dimensions of $V_i$ defined above.
Since $\{P_W(b): b \in B\}$ is complete in $W$ and $\text{dim}(W) = \gamma$ it is clear that we can choose $\{b_1, \dots, b_\gamma\} \subset B$ such that  $\{P_W(b_i): i = 1, \dots, \gamma\}$ is a basis of $W$. Since $\{w_1, \dots, w_\gamma\}$ is also a basis of $W$, there exist unique
scalars $\{\theta_{i,j}: i,j = 1,\dots,\gamma\}$ such that, 
\begin{equation} \label{**}
 P_W(b_i) = \sum_{j=1}^\gamma \theta_{ij}w_j.
 \end{equation}
with the matrix  $\Theta = \{\theta_{i,j}\}_{i,j = 1, \dots, \gamma}$ invertible.
 Thus, using Lemma~\ref{perfect} we can relabel the indices of $\{b_i\}$  in such a way that $\Theta$ is {\em \perfect}. Therefore, without loss of generality, we can assume that $\{b_1, \dots, b_\gamma\}$ are already in the right order, so that $\Theta$ is \perfect.
 
We will now prove that the $d$ vectors $\left\{b_i, Nb_i, \dots, N^{t_i -1}b_i\right\}_{i=1,\dots, \gamma}$ are linearly independent. For this, assume that there exist scalars $\alpha_j^s$ such that
 \begin{equation} \label{++}
 0 = \sum_{j=1}^{\gamma}\alpha_j^0 b_j + \sum_{j=1}^{p_1} \alpha_j^1 Nb_j + \dots + \sum_{j=1}^{p_{r-1}}
 \alpha^{r-1}_j N^{r-1}b_j,
\end{equation}
where $p_s = \max\{j: t_j > s\}=\text{dim}N^{s}W, s=1, \dots, r-1$ (note that $p_s \geq 1,$ since $N^{r-1}b_1 \not=0$).

Note that since $V = W \oplus NW \oplus \dots \oplus N^{r-1}W$, for any vector $x \in V$, $P_W(Nx) = 0$.
Therefore, if we apply $P_W$ on both sides of \eqref{++}, we obtain
\begin{equation*}
\sum_{j=1}^\gamma \alpha^0_j P_W b_j = 0.
\end{equation*}
Since $\{P_Wb_i:i=1,\dots,\gamma\}$ are linearly independent, we have $\alpha^0_j = 0, \;j = 1, \dots, \gamma$. Hence, if we now apply $P_{NW}$ to \eqref{++}, we have as before that
\begin{equation*}
\sum_{j=1}^{p_1} \alpha^1_j P_{NW} N b_j = 0.
\end{equation*}
Using the conmutation property of the projection, \eqref{nilp}, we have
\begin{equation*}
\sum_{j=1}^{p_1} \alpha^1_j N P_{W}  b_j = 0.
\end{equation*}
In matrix notation, this is
\begin{equation*}
[\alpha_1^1 \dots \alpha^1_{p_1}] \Theta_{p_1} \left[\begin{array}{c} Nw_1\\ \vdots \\Nw_{p_1}\end{array}\right] = 0.
\end{equation*}
Note that by definition of $p_1$, $Nw_1, \dots, Nw_{p_1}$ span $NW$,  and since the dimension of $NW$ is exactly $p_1$,  $Nw_1, \dots, Nw_{p_1}$ are linearly independent vectors. Therefore
$ [\alpha_1^1 \dots \alpha^1_{p_1}] \Theta_{p_1} = 0$. Since $\Theta$ is \perfect, $ [\alpha_1^1 \dots \alpha^1_{p_1}]  = [0 \dots 0]$.
Iterating the above argument, the Theorem follows.
\end{proof}

 \begin{proof}[Proof of Theorem \ref{Jordan-N} ]
 \ 
 
We will prove  the case when the minimal polynomial has only two factors. The general case follows by induction.

That is, let $T:V\to V$ be a linear transformation with characteristic polynomial of the form 
$\chi_T(x) = (x-\lambda_1)^{h_1}(x-\lambda_2)^{h_2}$. Thus, $V = V_1 \oplus V_2$ where $V_1, V_2$ are the subspaces associated to each factor,
and $T=T_1\oplus T_2$. In addition, $W=W_1\oplus W_2$ where $W_1, W_2$ are the subspaces  of the cyclic vectors from the cyclic decomposition of $N_1$ with respect of $V_1$ and
of $N_2$ with respect to $V_2$.

Let $\{b_i: i \in \Omega\} $ be vectors in $V$ that satisfy the hypothesis of the Theorem. For each $b_i$ we write $b_i = c_i + d_i$ with $c_i \in V_1$ and $d_i \in V_2$, $i \in \Omega$.  Let $r_i, m_i$ and $n_i$ be the degrees of the annihilators $q^{T}_{b_i}$, $q^{T_1}_{c_i}$ and $q^{T_2}_{d_i}$, respectively. By hypothesis  $\{P_{W_1}c_i: i\in \Omega\}$ and  $\{P_{W_2}d_i: i\in \Omega\}$ are complete in $W_1$ and $W_2$, respectively. Hence, applying Theorem \ref{Jordan-1block} to $N_1$ and $N_2$ we conclude that  $\bigcup_{i\in \Omega}\{T_1^jc_i,  j=0,1, \dots m_i-1\}$ is complete in $V_{1},$ and that $\bigcup_{i\in \Omega}\{T_2^jd_i,  j=0,1, \dots n_i-1\}$ is complete in $V_{2}$. 

We will now need a Lemma: (Recall that $q^T_b$ is the $T$-annihilator of the vector $b$)

\begin{lemma} \label{lem-u}Let $T$ be as above, and $V = V_1 \oplus V_2$. Given $b \in V$, $b=c+d$ then 
$q^T_b = q^{T_1}_c q^{T_2}_d$ where $q^{T_1}_c$ and $q^{T_2}_d$ are coprime. Further let $u \in V_2$, $u = q^{T_1}_c(T_2)d$. Then $q^{T_2}_u$ coincides with $q^{T_2}_d$.
\end{lemma}
\begin{proof}

The fact that $q^T_b = q^{T_1}_c q^{T_2}_d$ with coprime $q^{T_1}_c$ and $q^{T_2}_d$ is a consequence of the decomposition of $T$.

Now, by definition of $q^{T_2}_u$ we have that
$$ 0 = q^{T_2}_u(T_2)(u)= q^{T_2}_u(T_2)(q^{T_1}_c(T_2)d) = (q^{T_2}_u q^{T_1}_c)(T_2)d.$$
Thus, $q^{T_2}_d$ has to divide $q^{T_2}_u\cdot q^{T_1}_c$, but since $q^{T_2}_d$ is coprime with $q^{T_1}_c$, we conclude that
\begin{equation}\label{l1}
q^{T_2}_d \quad \text{divides} \quad q^{T_2}_u.
\end{equation}

On the other hand
\begin{align*}
 0 & = q^{T_2}_d(T_2)(d)  = q^{T_1}_c(T_2)(q^{T_2}_d(T_2)d) = (q^{T_1}_c q^{T_2}_d)(T_2)d  \\
  & = (q^{T_2}_d q^{T_1}_c)(T_2)d = q^{T_2}_d(T_2)(q^{T_1}_c(T_2)d) = q^{T_2}_d(T_2)(u), 
  \end{align*}
and therefore
\begin{equation}\label{l2}
q^{T_2}_u \quad \text{divides} \quad q^{T_2}_d.
\end{equation}
From \eqref{l1} and \eqref{l2} we obtain $q^{T_2}_d = q^{T_2}_u$.
\end{proof}

Now, we continue with the proof of the Theorem. Recall $r_i, m_i$ and $n_i$ be the degrees of $q^{T}_{b_i}$, $q^{T_1}_{c_i}$ and $q^{T_2}_{d_i}$, respectively, and let $\li = r_i-1$. Also note that by Lemma \ref {lem-u} $r_i=m_i+n_i$. In order to prove that the set $\{b_i, T b_i, \dots, T^{\li} b_i: i \in \Omega\}$ is complete in $V,$ we will replace this set with a new one in such a way that the dimension of the span does not change.

For each $i  \in \Omega$, let $u_i = q^{T_1}_{c_i}(T_2)d_i$. Now, for a fixed $i$ we leave the vectors
$b_i, T b_i, \dots, T^{m_i-1}b_i$ unchanged, but for $s=0, \dots, n_i-1$ we replace the vectors $T^{m_i+s}b_i$ by the vectors $T^{m_i+s}b_i + \beta_s(T)b_i$
where $\beta_s$ is the polynomial $\beta_s(x) = x^s q^{T_1}_{c_i}(x) - x^{m_i+s}.$

Note that $\text{span}\{b_i, T b_i, \dots, T^{m_i+s}b_i\}$ remains unchanged, since $\beta_s(T)b_i$ is a linear combination of the vectors
$\{T^sb_i, \dots , T^{m_i+s-1}b_i\}$.

Now we observe that:
$$
T^{m_i+s}b_i + \beta_s(T) b_i = \left[ T_1^{m_i+s}c_i  +\beta_s (T_1)c_i\right]   +   \left[ T_2^{m_i+s}d_i  +\beta_s(T_2) d_i\right] .
$$
The first term of the sum on the right hand side of the equation above is in $V_1$ and the second in $V_2$.
By definition of $\beta_s$ we have:

$$
T_1^{m_i+s}c_i  +\beta_s (T_1)c_i = T_1^{m_i+s}c_i  + T_1^s q^{T_1}_{c_i}(T_1)c_i - T_1^{m_i+s} c_i = T_1^s q^{T_1}_{c_i}(T_1)c_i = 0, 
$$
 and  
$$
T_2^{m_i+s}d_i  +\beta_s (T_2) d_i  = T_2^sq^{T_1}_{c_i}(T_2)(d_i) = T_2^s u_i.
$$
Thus, for each $i \in \Omega$, the vectors $\{b_i, \dots, T^{\li}b_i\} $ have been replaced by
the vectors $\{b_i, \dots, T^{m_i-1}b_i, u_i, \dots, T^{n_i-1}u_i \}$ and both sets have the same span.

To finish the proof we only need to show that the new system is complete in $V$.

Using Lemma \ref{lem-u}, we have that for each $i \in \Omega$,
 $$\text{dim(span}\{ u_i,\dots, T_2^{n_i-1}u_i\}) = \text{dim(span}\{ d_i,\dots, T_2^{n_i-1}d_i\})= n_i,$$
 
 and since each $T_2^su_i \in \text{span}\{ d_i,\dots, T_2^{n_i-1}d_i\}$ we conclude that 
\begin{equation} \label{span-u}
\text{span}\{ u_i,\dots, T_2^{n_i-1}u_i: i\in \Omega\} = \text{span}\{ d_i,\dots, T_2^{n_i-1}d_i: i  \in \Omega\}.
\end{equation}
 
 Now assume that $x\in V$ with $x=x_1+x_2, \,\, x_i\in V_i.$ Since by hypothesis $\text{span}\{ c_i,\dots, T_1^{m_i-1}c_i: i\in \Omega\}$ is complete in $V_1$, we can write 
  \begin{equation}\label{comp-1}
x_1 = \sum_{i\in \Omega} \sum_{j=0}^{m_i-1} \alpha^i_j T_1^{j}c_i  , 
\end{equation}
for same scalars $\alpha^i_j$, and therefore, 
 \begin{equation}\label{comp-2}
  \sum_{i\in \Omega} \sum_{j=0}^{m_i-1} \alpha^i_j T^{j}b_i = x_1 +  
  \sum_{i\in \Omega} \sum_{j=0}^{m_i-1} \alpha^i_j T_2^{j}d_i = x_1 + \tilde x_2, 
  \end{equation}
 since $\sum_{i\in \Omega} \sum_{j=0}^{m_i-1} \alpha^i_j T_2^{j}d_i =\tilde x_2$ is in $V_2$ by the invariance of $V_2$ by $T$. Since by hypothesis $\{T_2^j d_i : i\in \Omega, \quad j=1, \dots,n_i-1\}$ is complete in $V_2$,  by equation  \eqref{span-u}, 
 $\{T_2^j u_i:i\in \Omega, \quad j=1, \dots, n_i-1\}$  is also complete in $V_2$, and therefore there exist scalars $\beta^i_j$, 
 $$
x_2 - \tilde x_2 = \sum_{i\in \Omega} \sum_{j=0}^{n_i-1} \beta^i_j T_2^{j}u_i,
 $$ 
 and so
 $$
 x =  \sum_{i\in \Omega} \sum_{j=0}^{m_i-1} \alpha^i_j T^{j}b_i + \sum_{i\in \Omega} \sum_{j=0}^{n_i-1} \beta^i_j T_2^{j}u_i,
 $$ 
 which completes the proof of Theorem \ref{Jordan-N} for the case of two coprime factors
   in the minimal polynomial of $J$. The general  case of more factors follows by induction 
   adapting the previous argument.

\end{proof}

Theorem \ref{tadcul} and Theorem \ref{tacul} and its corollaries are easy consequences of Theorem \ref{Jordan-N}.

\begin{proof}[Proof of Theorem \ref{tacul}]
Note that if $\{J^{L+1}b_i: i \in \Omega\} \subset \text{span} (E)$, then $\{J^{L+2}b_i: i \in \Omega\} \subset \text{span} (E)$ as well. Continuing in this way, it follows that for each $i \in \Omega$, $span (E)$ contains all the powers $J^jb_i$ for any $j$. Therefore, using Theorem \ref{tadcul}, it follows that 
$span (E) $ contains a frame of $\C^d$, so that, $span (E)=\C^d$ and $E$ is a frame of $\C^d.$ The converse is obvious.
\end{proof}
The proof of Theorem \ref{TAL} uses a similar argument.

Although Theorem \ref {TAD0} is a direct consequence of Theorem \ref {tadcul}, we will give a simpler proof for this case. 

\begin{proof} [Proof of Theorem \ref {TAD0}]
\ 

Let $\{P_j(b_i):i\in \Omega\}$ form a frame of $P_j(\C^d)$, for each $j=1,\dots ,n$. Since we are working with finite dimensional spaces, to show that $\{D^{j}b_i: \; i\in \Omega,  \, j=0, \dots, \li\}$ is a frame of $\C^d$, all we need to show is that it is complete in $\C^d$. Let $x$ be any vector in $\C^d$, then $x=\sum\limits_{j=1}^n P_j x$. Assume that $\langle D^lb_i,x\rangle=0$ for all $i\in \Omega$ and $l=0,\dots,l_i$. Since $l_i=r_i-1$, where $r_i$ is the degree of the $D$-annihilator of $b_i$, we have that $\langle D^lb_i,x\rangle=0$ for all $i\in \Omega$ and $l=0,\dots,d$. In particular, since $n\le d$, $\langle D^lb_i,x\rangle=0$ for all $i\in \Omega$ and $l=0,\dots,n$. Then
\begin{equation}
\label {compcd}
\langle D^lb_i,x\rangle=\sum\limits_{j=1}^n\langle D^lb_i,P_jx\rangle=\sum\limits_{j=1}^n\lambda^l_j\langle P_jb_i,P_jx\rangle=0,
\end{equation}
for all $i\in \Omega$ and $l=0,\dots,n$. 
Let $z_i$ be the vector $\big(\langle P_jb_i,P_jx\rangle\big) \in \C^n$. Then for each $i$, \eqref {compcd} can be written in matrix form as $Vz_i=0$ where V is the $n\times n$ Vandermonde matrix
\begin {equation}\label {VM}
V=
 \begin{pmatrix}
 
1 & 1 & \cdots & 1 \\
  \lambda_1 & \lambda_2 & \cdots & \lambda_n \\
  \vdots  & \vdots  & \ddots & \vdots  \\
  \lambda_1^{n-1} & \lambda_2^{n-1} & \cdots & \lambda_n^{n-1}
 \end{pmatrix},
 \end {equation}
which is invertible since, by assumption, the $\lambda_j$s are distinct. Thus, $z_i=0$. Hence, for each $j$, we have that $\langle P_jb_i,P_jx\rangle=0$ for all $i\in \Omega$. Since $\{P_j(b_i):i\in \Omega\}$ form a frame of $P_j(\C^d)$, $P_jx=0$. Hence, $P_jx=0$ for $j=1,\dots,n$ and therefore $x=0$. 

\end{proof}

\subsection{Remark}

Given a general linear transformation $T:V\longrightarrow V$, the cyclic decomposition theorem gives the rational form
for the matrix of $T$ in some special basis.  A natural question is then if we can obtain a similar result to Theorem \ref{Jordan-N} for this decomposition.
(Rational form instead of Jordan form). The answer is no. That is, if a set of vectors $b_i$ with $i \in \Omega$ where $\Omega$ is a finite subset of $\{1, \dots, d\}$ 
when projected onto the subspace generated by the cyclic vectors, is complete in this subspace, this does not necessarily imply 
that its iterations $T^jb_i$ are complete in $V$.
The following example illustrates this fact for a single cyclic operator.
\begin{itemize}
\item
Let $T$ be the linear transformation in $\R^3$ given as multiplication by the following matrix $M.$
$$M = \left[ \begin{array}{rrr}
0&0&1\\
1&0&1\\
0&1&2 \end{array}\right]
$$
The matrix $M$ is in rational form with just one cyclic block. 
The vector $e_1 = (1,0,0)$ is cyclic for $M$.  However it is easy to see that there exists a vector $b=\left[\begin{array}{r}x_1\\x_2\\x_3\end{array}\right]$ in $\R^3$ such that $P_W(b)=x_1 \not=0$, (here $W$ is span$\{e_1\}$), but 
$\{b, Mb, M^2b\}$ are linearly dependent, and hence do not span $\R^3$.
So our proof for the Jordan form uses the fact that the cyclic components in the Jordan decomposition are nilpotent!
\end{itemize}


\section{Dynamical Sampling in infinite dimensions}
\label {DSID}
In this section we consider the dynamical sampling problem in a separable Hilbert space $\mathcal H$, that without any lost of generality we can consider to be $\ell^2(\N)$. The evolution operators we will consider belong to the following  class $\DO$ of bounded self adjoint operators:
$$ \mathcal {A}=\{ A \in \mathcal{B}(\ell^2(\N)):  A=A^\ast, \text{and there exists a basis of $ \ell^2(\N) $  of eigenvectors of } A \}.$$ 
The notation $\mathcal{B}(\mathcal H)$ stands for the bounded linear operators on the Hilbert space $\mathcal H$.
So, if $A\in\DO$  there exists an unitary operator $B$ such that $A=B^*DB$ with  $D=\sum_j\lambda_jP_j$ with pure spectrum $\sigma_p(A)=\{\lambda_j: j\in \N\}\subset \R$ and orthogonal projections $\{P_j\}$ such that $\sum_jP_j=I$ and $P_jP_k=0$ for $j \ne k$.
Note that the class $\DO$ includes all the bounded self-adjoint compact operators.

\begin{remark}\label{avsd}
Note that by the definition of $\DO$, we have that for any $f \in \ell^2(\N)$ and $l=0, \dots $ \[<f,A^le_j>=<f,B^{\ast}D^lB e_j>=<B f, D^lb_j> \quad \text{and} \quad \|A^l\| = \|D^l\|.
\]
It follows that $\X_\Omega=\big\{A^le_i:\; i \in \Omega, l=0,\dots,l_i\big\}$ is complete, (minimal, frame) if and only if $\big\{ D^lb_i:\, i \in \Omega, l=0,\dots,l_i\big\}$ is complete (minimal, frame).
\end{remark}

\subsection {Completeness} \label {completeness}
In this section, we characterize the sampling sets $\Omega \subset \N$ such that a function $f\in \ell^2(\N)$ can be recovered from the data 
\[
Y=\{f(i),Af(i), A^2f(i),\dots, A^{l_i}f(i): \, i \in \Omega\}
\]
where $A \in \DO$, and  $0\le l_i\le \infty$.

 For  each set $\Omega$ we consider the set of vectors $O_\Omega:=\{b_j= Be_j: j \in \Omega\}$, where  $e_j$ is the $j$th canonical vector of $\ell^2(\N)$. For each $b_i \in O_\Omega$ we define $r_i$ to be the degree of the $D$-annihilator of $b_i$ if such annihilator exists, or we set $r_i=\infty$. This number $r_i$ is also  the degree of the $A$-annihilator of $e_i$. for the remainder of this paper we let $l_i=r_i-1$.

\begin {theorem} \label {complete}
Let $A\in \DO$ and $\Omega\subset \N$. Then the set $\X_\Omega=\big\{A^le_i:\; i \in \Omega,  l=0,\dots, l_i\big\}$ is complete in $\ell^2(\N)$ if and only if for each $j$, the set $\big\{P_j(b_i): i  \in \Omega\big\}$ is complete on the range $E_j$ of $P_j$. 

In particular, $f$ is determined uniquely from the set 
\[
Y=\{f(i),Af(i), A^2f(i),\dots, A^{l_i}f(i): \, i \in \Omega\}
\]   
if and only if for each $j$, the set $\big\{P_j(b_i): i  \in \Omega\big\}$ is complete in the range $E_j$ of $P_j$.
\end{theorem}
\begin{remarks} \label {GenClass} 
\item i) Note that Theorem \ref {complete} implies that $|\Omega|\ge \sup_j \dim (E_j)$. Thus, if any eigen-space has infinite dimensions, it is necessary to have infinitely many ``spacial'' sampling points in order to recover $f$.
\item ii) Theorem \ref {complete} can be extended to a larger class of operators. For example, for the class of operators $\widetilde {\mathcal A}$ in $\mathcal B( \ell^2(\N))$ in which $A \in \widetilde {\mathcal A}$ if $A=B^{-1}DB$ where with  $D=\sum_j\lambda_jP_j$ with pure spectrum $\sigma_p(A)=\{\lambda_j: j\in \N\}\subset \C$ and orthogonal projections $\{P_j\}$ such that $\sum_jP_j=I$ and $P_jP_k=0$ for $j \ne k$.
 
\end{remarks}
\begin {proof}[Proof of Theorem~\ref{complete}]
\ 

By Remark~\ref{avsd}, to prove the theorem we only need to show that $\big\{ D^lb_i:\, i \in \Omega, l=0,\dots,l_i\big\}$ is complete if and only if  for each $j$, the set $\big\{P_j(b_i): i  \in \Omega\big\}$ is complete in the range $E_j$ of $P_j$. 

Assume that $\big\{ D^lb_i:\, i \in \Omega, l=0,\dots,l_i\big\}$ is complete. For a fixed $j$, let $g\in E_j$ and assume that $<g,P_jb_i>=0$ for all $i\in \Omega$. Then for any $l=0,1,\dots,l_i$,  we have
\[\lambda^l <g,P_jb_i>= <g,\lambda^lP_jb_i>=<g,P_jD^lb_i>=<g,D^lb_i>=0.\]
Since $\big\{ D^lb_i:\, i \in \Omega, l=0,\dots,l_i\big\}$ is complete in $\ell^2(\N)$, $g=0$. It follows that $\big\{P_j(b_i): i  \in \Omega\big\}$ is complete on the range $E_j$ of $P_j$.

Now assume that $\big\{P_j(b_i): i  \in \Omega\big\}$ is complete in the range $E_j$ of $P_j$. Let $S=\overline{span} \big\{D^lb_i;  i \in \Omega, l=0,\dots,l_i\big\}$. Clearly $DS\subset S$. Thus $S$ is invariant for $D$. Since $D$ is self-adjoint, $S^\perp$ is also invariant for $D$. It follows that the orthogonal projection $P_{S^\perp}$ commutes with $D$. Hence, $P_{S^\perp}=\sum_jP_jP_{S^\perp}P_j$ where convergence is in the strong operator topology. In particular $P_{S^\perp} b_i=\sum_jP_jP_{S^\perp}P_jb_i=0$. Multiplying both sides by the projection $P_k$ for some fixed $k$ we get that
\[
P_kP_{S^\perp}P_kb_i= P_kP_{S^\perp}P_k(P_kb_i)=0.
\]
Since $\big\{P_k(b_i): i  \in \Omega\big\}$ is complete in $E_k$ and since $k$ was arbitrary, it follows that $P_kP_{S^\perp}P_k=0$ for each $k$. Hence $P_{S^\perp}=0$. That is $\X_\Omega$ is complete which finishes the proof of the theorem.
\end {proof}
\subsection{Minimality and bases for the dynamical sampling in infinite dimensional Hilbert spaces}
\ 

In this section we will show, that for any $\Omega\subset \N$, the set  $\X_\Omega=\big\{A^le_i:\; i \in \Omega, l=0,\dots,l_i \big\}$ is never minimal  if $\Omega$ is a finite set, and hence the set $\X_\Omega$ is never a basis. In some sense, the set $\X_\Omega$ contains many "redundant vectors"  which prevents it from being a basis. However, when $\X_\Omega$ is complete, this redundancy may help $\X_\Omega$ to be a frame. We will discuss this issue in the next section. For this section, we need the celebrated M\"untz-Sz\'asz Theorem characterizing the sequences of monomials that are complete in $C[0,1]$ or $C[a,b]$ \cite {Heil11}:

\begin {theorem}[M\"untz-Sz\'asz Theorem] \label {MS}
Let $0\le n_1\le n_2\le \dots$ be an increasing sequence of nonnegative integers. Then
\begin{enumerate}
\item $\{x^{n_k}\}$ is complete in $C[0,1]$ if and only if $n_1=0$ and $\sum\limits_{k=2}^\infty 1/n_k=\infty$.
\item If $0<a<b< \infty$, then $\{x^{n_k}\}$ is complete in $C[a,b]$ if and only if  $\sum\limits_{k=2}^\infty 1/n_k=\infty$.
\end{enumerate}
\end {theorem}
We are now ready to state the main results of this section. 
\begin {theorem}\label {MIN}
Let $A\in \DO$ and let $\Omega$ be a non-empty subset of $\N$. If there exists $b_i\in O_\Omega$ such that $r_i=\infty$, then the set $\X_\Omega$ is not minimal.
\end {theorem}

As an immediate corollary we get
\begin {theorem}\label {MINFiniteOmega}
Let $A\in \DO$ and let $\Omega$ be a finite subset of $\N$. If $\X_\Omega=\big\{A^le_i:\; i \in \Omega, l=0, \dots, l_i \big\}$ is complete in $\ell^2(\N)$, then $\X_\Omega$ is not minimal in $\ell^2(\N)$.
\end {theorem}

Another immediate corollary is
\begin {theorem}\label {NonBasisFiniteOmega}
Let $A\in \DO$ and let $\Omega$ be a finite subset of $\N$. Then the set $\X_\Omega=\big\{A^le_i:\; i \in \Omega, l=0, \dots, l_i \big\}$ is not a basis for $\ell^2(\N)$.
\end {theorem}

\begin{remarks} \label {GenClass2}
\ 

\begin{enumerate}
\item Theorem \ref {NonBasisFiniteOmega} remains true for the class of operators $A \in \widetilde {\mathcal A}$ described in Remark \ref {GenClass}.

 \item Theorems \ref {MINFiniteOmega} and \ref {NonBasisFiniteOmega} do not hold in the case of $\Omega$ being an infinite set. A trivial example is when $A=I$ is the identity matrix and $\Omega=\N$. A less trivial example is when   $B\in \ell^2(\Z)$ is the symmetric bi-infinite matrix with entries $B_{ii}=1$, $B_{i(i+1)}=1/4$ and $B_{i(i+k)}=0$ for $k\ge 2$. Let $\Omega=3\Z$ and $D_{kk}=2$ if $k=3\Z$, $D_{kk}=1$ if $k=3\Z+1$, and $D_{kk}=-1$ if $k=3\Z+2$. Then $\X_\Omega=\big\{A^le_i:\; i \in \Omega, l=0, \dots, 2 \big\}$ is a basis for $\ell^2(\Z)$. In fact $\X_\Omega$ is a Riesz basis of $\ell^2(\Z)$. Examples in which the $\Omega$ is nonuniform can be found in \cite{ADK14}. 
\end{enumerate}
\end{remarks}

\begin {proof} [Proof of Theorem \ref {MIN}] 
\ 

Again, using Remark~\ref{avsd}, we will show that $\{D^lb:l =0,1,\dots\}$ is not minimal. We first assume that $D=\sum_j\lambda_jP_j$ is non-negative, i.e., $\lambda_j\ge 0$ for all $j \in \N$. Since $A\in B(\ell^2(\N))$, we also have that $0\le\lambda_j\le \|D\|< \infty$. Let $b \in O_\Omega$ with $r=\infty$ and $f\in \overline{span} \{D^lb:l =0,1,\dots\}$ be a fixed vector. Let $n_k$ be any increasing sequence of nonnegative integers  such that $\sum\limits_{k=2}^\infty 1/n_k=\infty$.  Then for any $\epsilon>0$, there exists a polynomial $p$ such that $\|f-p(D)b\|_2\le \epsilon/2$. Since the polynomial $p$ is a continuous function on $C[0,\|D\|]$, (by the M\"untz-Sz\'asz Theorem) there exists  a function $g \in span \{1,x^{n_k}: k \in \N\}$ such that $\sup \big \{|p(x)-g(x)|:\, x \in [0,\|D\|] \big\}\le \frac {\epsilon}{2 \|b\|_2} $. Hence
\begin{eqnarray*}
 \|f-g(D)b\|_2&\le&\|f-p(D)b\|_2+ \|p(D)b-g(D)b\|_2\\
 &\le& \frac{\epsilon}{2} + \frac{\epsilon}{2\|b\|_2}\|b\|_2=\epsilon.
 \end{eqnarray*}
 Therefore $\overline {span} \{b,D^{n_k}b: k \in \N\}=\overline{span} \{D^lb:l =0,1,\dots\}$ and we conclude that $\{D^lb:l =0,1,\dots\}$ is not minimal.
 
 If the assumption about the non-negativity of $D=\sum_j\lambda_jP_j$ is removed, then by the previous argument $ \{D^{2l}b:l =0,1,\dots\}$ is not minimal hence $\{D^lb:l =0,1,\dots\}$ is not minimal either, and the proof is complete.

\end {proof}

\subsection{Frames in infinite dimensional Hilbert spaces}
\ 

In the previous sections, we have seen that although the set $\X_\Omega=\big\{A^le_i:\; i \in \Omega,  l=0,\dots,l_i\big\}$ is complete for  appropriate sets $\Omega$, it cannot  form a basis for $\ell^2(\N)$ if $\Omega$ is a finite set, in general. The main reason is that $\X_\Omega$ cannot be minimal, which is necessary to be a basis. On the other hand, the non-minimality is a statement about redundancy. Thus, although $\X_\Omega$ cannot be a basis, it is possible that $\X_\Omega$ is a frame for sets $\Omega\subset \N$ with finite cardinality. Being a frame is in fact desirable since in this case we can reconstruct any $f \in \ell^2(\N)$ in  stable way from the data $Y=\{f(i),Af(i), A^2f(i),\dots, A^{l_i}f(i): \, i \in \Omega\}$. 

In this section we will show that, except for some special case of the eigenvalues of $A$,  if $\Omega$ is a finite set, i.e., $|\Omega|<\infty$, then $\X_\Omega$ can never be a frame for $\ell^2(\N)$. Thus essentially, either the eigenvalues of $A$ are {\em nice} as we will make precise below and even a single element set $\Omega$ and its iterations may be a frame, or, the only hope for $\X_\Omega$ to be be a frame for $\ell^2(\N)$ is that $\Omega$ infinite - moreover, it need to be well-spread over $\N$. 

\begin {theorem} \label {Noframe0}
Let $A\in \DO$ and let $\Omega\subset \N$ be a finite subset of $\N$. If  $\X_\Omega=\big\{A^le_i:\; i \in \Omega,  l=0,\dots, l_i\big\}$ is a frame, then 
\[
\inf \{\|A^le_i\|_2: \; i \in \Omega,  l=0,\dots,l_i\}=0.
\]

\end{theorem}
\begin{proof}[Proof of Theorem \ref {Noframe0}]
\ 

 If $\X_\Omega$ is  a frame, then it is complete. Thus the set  $\Omega_\infty:=\{i\in \Omega: l_i=\infty\}$ is nonempty since $|\Omega|< \infty$. In addition, if $\X_\Omega$ is a frame, then it is a Bessel sequence. Hence $\sup \{\|(D)^lb_i\|_2=\|A^le_i\|_2: i \in \Omega,  l=0,\dots,l_i\} \le C_2$ for some $C_2<\infty$. 
 
 If $\inf \{\|A^le_i\|_2: \; i \in \Omega,  l=0,\dots,l_i\}\ne0$, then the set $\{\|(D)^lb_i\|_2=\|A^le_i\|_2: i \in \Omega,  l=0,\dots,l_i\}$ is also bounded below by some positive constant $C_1$. Therefore, the Kadison-Singer/Feichtinger conjectures proved recently \cite{MSS13} applies, and $\X_\Omega$ is the finite union of Riesz sequences $\bigcup\limits_{j=1}^N R_j$. Let $i_0 \in \Omega_\infty$. Consider the subset $\{(D^2)^lb_{i_0} : l=0,\dots,\infty\}$. Then (by Remark~\ref{avsd}) each set $\{(D^2)^lb_{i_0} : l=0,\dots,\infty\}\cap R_j$ is a Riesz  sequence as well (i.e., a Riesz basis of the closure of its span).  Thus, there exists an increasing sequence $\{n_k\}$  with $\sum\limits_{k=2}^\infty\frac{1}{n_k}=\infty$ and $j_0$ such that $\{(D^2) ^{n_k}b_{i_0}: k=1,\dots,\infty\}$ is a subset of $R_{j_0}$. Hence  $ \{(D^2)^{n_k}b_{i_0}: k=1,\dots,\infty\}$ is a basis of 
$\overline{span} \{(D^2)^{n_k}b_{i_0}: k=1,\dots,\infty\}$. Note that $\sigma(D^2)\subset [0,\|D\|^2]$. Now using the M\"untz-Sz\'asz Theorem \ref {MS} it follows $\{(D^2)^{n_k}b_{i_0}: k=1,\dots,\infty\}$ is not minimal and hence is not a basis for the closure of its span which is a contradiction. Thus, $\inf \{\|A^le_i\|_2: \; i \in \Omega,  l=0,\dots,l_i\}=0$. 
\end{proof}
Therefore, when $|\Omega|<\infty$, the only possibility for $\X_\Omega$ to be a frame, is that 
$$  \inf \{\|A^le_i\|_2: \; i \in \Omega,  l=0,\dots,l_i\} = 0$$ 
and
$$ \sup \{\|A^le_i\|_2: \; i \in \Omega,  l=0,\dots,l_i\} \leq C_2 < \infty. $$
We have the following theorem to establish for which finite sets $\Omega$,  $\X_\Omega$ is not a frame for $\ell^2(\N)$.
\begin {theorem} \label {Noframe1}
Let $A\in \DO$ and let $\Omega\subset \N$  be a finite subset of $\N$.  
 If $1$ or $-1$ is(are) not a cluster point(s) of   $\sigma(A)$, then $\X_\Omega=\big\{A^le_i:\; i \in \Omega,  l=0,\dots,l_i\big\}$ is not a frame.
\end{theorem}
As an immediate corollary we get
\begin {corollary} \label {Noframecompact}
Let $A$ be a compact self-adjoint operator, and $\Omega\subset \N$ be a finite set.  
 Then $\X_\Omega=\big\{A^le_i:\; i \in \Omega,  l=0,\dots,l_i\big\}$ is not a frame.
\end{corollary}
\begin{remark}
Theorems \ref {Noframe0} and \ref {Noframe1} can be generalized for the class $A \in \widetilde {\mathcal A}$.
\end{remark}
\begin{proof}[ Proof of Theorem \ref {Noframe1}]
\ 

If $\X_\Omega$ is not complete, then it is obviously not a frame of $\ell^2(\N)$. If $\X_\Omega$ is complete in $\ell^2(\N)$, then  the set $\Omega_\infty:=\{i\in \Omega: l_i=\infty\}$ is nonempty.    

If there exists $j \in \N$ and $i \in \Omega_\infty$ such that $|\lambda_j|\ge1$ and $P_jb_i\ne 0$ then for $x=P_jb_i$ we have 
\[
\sum_l|\langle x, D^l b_i\rangle |^2=\sum_l|\lambda_j|^{2l}\|P_jb_i\|_2^4=\infty.
\]
Thus, $\X_\Omega$ is not a frame.

Otherwise, let $r=\sup\limits_{j \in \N}\{|\lambda_j|:  P_jb_i \ne 0 \text{ for some } i \in \Omega_\infty\}<1$. 

Since $-1$ or $1$ are not cluster points of $\sigma(A)$, $r < 1$. But
\[
\|Db_i\|_2\le \sup\limits_{j \in \N}\{|\lambda_j|:  P_jb_i \ne 0\}\|b_i\|_2 \quad \forall i\in \Omega_\infty,
\]
and  therefore we have that  $\|D^lb_i\|_2\le r^l\|b_i\|_2$. 
Now given $\epsilon>0$, there exists $N$ such that
\[
\sum\limits_{i\in\Omega_\infty}\sum\limits_{l>N}\|D^lb_i\|^2_2\le \epsilon.
\]
Choose $f\in \ell^2(\N)$ such that $\|f\|_2=1$, $<f,D^lb_i>=0$ for all $i\in\Omega-\Omega_\infty$ and $l=0,\dots,l_i$ and such that  $<f,D^lb_i>=0$ for all $i\in\Omega_\infty$ and $l=0,\dots,N$. Then
\[
\sum\limits_{i\in\Omega}\sum\limits_{l=0}^{l_i}|<f,D^lb_i>|^2\le\epsilon = \epsilon \|f\|_2.
\]
Since $\epsilon$ is arbitrary, the last inequality implies that $\X_\Omega$ is not a frame since it cannot have a positive lower frame bound.
\end{proof}

Although Theorem \ref {Noframe1} states that $\X_\Omega$ is not a frame for $\ell^2(\N)$, it could be that after normalization of the vectors in $\X_\Omega$, the new set $\mathcal Z_\Omega$  is a frame for $\ell^2(\N)$. It turns out that the obstruction is intrinsic. In fact, this case is even worse, since $\mathcal Z_\Omega$ is not a frame even if $1$ or $-1$ is (are) a cluster point(s) of $\sigma(A)$.

\begin {theorem} \label{NotFrame2} 
Let $A\in \DO$ and let $\Omega\subset \N$  be a finite set.   Then the unit norm sequence $\big\{\frac {A^le_i}{\|A^le_i\|_2}:\; i \in \Omega,  l=0,\dots,l_i\big\}$ is not a frame.
\end{theorem}
\begin {proof}
Note that by Remark~\ref{avsd},
$\big\{\frac {A^le_i}{\|A^le_i\|_2}:\; i \in \Omega,  l=0,\dots,l_i\big\}$ is a frame if and only if $\mathcal Z_\Omega=\big\{\frac {D^lb_i}{\|D^lb_i\|_2}:\; i \in \Omega,  l=0,\dots,l_i\big\}$ is a frame.

Now, using the exact same argument as for Theorem~\ref{Noframe0} we obtain the desired result.
\end{proof}

We will now concentrate on the case where there is a cluster point of $\sigma(A)$ at $1$  or $-1$, and we start with the case where $\Omega$ consists of a single sampling point. 

Since $A \in \DO$, $A = B^*DB$, by Remark~\ref{avsd}  $\X_\Omega$ is a frame of $\ell^2(\N)$ if and only if there exists a vector $b = Be_j$ for some $j \in \N$ that corresponds to the sampling point, and $\{D^lb: l=0, \dots\}$ is a frame for $\ell^2(\N)$.

 For this case, Theorem \ref {complete} implies that if $\X_\Omega$ is a frame of $\ell^2(\N)$, then the projection operators $P_j$ used in the description of the operator $A\in \mathcal A$ must be of rank $1$. Moreover, the vector $b$ corresponding to the sampling point must have infinite support, otherwise $l_b<\infty$ and $\X_\Omega$ cannot be complete in $\ell^2(\N)$. Moreover, for this case in order for $\X_\Omega$ to be a frame, it is necessary that $|\lambda_k|<1$ for all $k$, otherwise, if there exists $\lambda_{j_0}\ge 1$ then for $x=P_{j_0}b$ 
(note that by Theorem \ref {complete} $P_{j_0}b\ne 0$) we would have  
\[
\sum_n|\langle x, D^n b\rangle |^2=\sum_n|\lambda_{j_0}|^{2n}\|P_{j_0}b\|_2^4=\infty,
\]
which is a contradiction. 

In addition, if $\X_\Omega$ is a frame, then the sequence $\{\lambda_k\}$  cannot have a cluster point  $a$ with $|a|<1$. To see this, suppose there is a subsequence $\lambda_{k_s} \rightarrow a$ for some $a$ with $|a|<1$, and let $W=\sum_sP_{k_s}(\ell^2(\N))$.  Then $W$ and $W^{\perp}$ are invariant for $D$, and $\ell^2(\N)=W\oplus W^{\perp}$. Set $D_1=D|_{W}$, and $b_1=P_Wb$ where $P_W$ is the orthogonal projection on $W$. Then, by 
 Theorem \ref{Noframe1}, $\{D_1^jb_1: j=0,1, \dots\}$ can not be a frame for  $ W$. It follows that $\X_\Omega$ cannot be a frame for $\ell^2(\N)=W\oplus W^{\perp}$.

Thus the only possibility for $\X_\Omega$ to be a frame of $\ell^2(\N)$ is that $|\lambda_k|\to 1.$ These remarks allow us to characterize when $\X_\Omega$ is a frame for the situation when $|\Omega|=1$.

\begin{theorem} \label {OnePointFrame}
Let  $D=\sum_j\lambda_jP_j$ be such that $P_j$ have rank $1$ for all $j\in \N$, and let $b \in \ell^2(\N)$. 
Then $\{D^lb: l=0,1,\dots\}$ is a frame if and only if
\begin{enumerate}
\item[i)] $|\la_k| < 1$  for all $k.$  
\item[ii)]  $|\la_k| \to 1$.
\item[iii)] $\{\lambda_k\}$ satisfies Carleson's condition 
\begin{equation}
\label{carleson-cond}
\inf_{n} \prod_{k\neq n} \frac{|\la_n-\la_k|}{|1-\bar{\la}_n\la_k|}\geq \delta.
\end{equation}
for some $\delta>0$.
\item[iv)] $b_k=m_k\sqrt{1-|\lambda_k|^2}$ for some sequence $\{m_k\}$ satisfying $0<C_1\le |m_k| \le C_2< \infty$.

\end{enumerate}
\end{theorem}

This theorem implies the following Corollary:

\begin{corollary} Let $A = B^*DB \in \DO$, and $D=\sum_j\lambda_jP_j$ be such that $P_j$ have rank $1$ for all $j\in \N$. Then, there exists $i_0\in \N$ such that $\X_\Omega = \{A^l e_{i_0}: l=0,\dots \}$ is a frame for $\ell^2(\N)$, if and only if $\{\lambda_j\}$ satisfy the conditions of Theorem \ref {OnePointFrame} and there exists $i_0\in \N$, such that   $b=Be_{i_0}$ satisfies the condition ${\rm iv}$ of Theorem \ref {OnePointFrame} . 
\end{corollary}

Theorem \ref {OnePointFrame} follows from the discussion above and the following two Lemmas

\begin{lemma}\label{equiv}
Let $D$ be as in Theorem \ref {OnePointFrame} and assume that $|\la_k| < 1$  for all $k$. Let $b^0_k=\sqrt{1-|\lambda_k|^2}$, and assume that $b^0\in \ell^2(\N)$. Let $b\in \ell^2(\N)$.

Then, $\{D^l b: l\in \N\}$ is a frame for $\ell^2(N)$ if and only if $\{D^l b^0: l\in \N\}$ is a frame and  there exist $C_1$ and $C_2$ such that $b_k/ b^0_k = m_k$ satisfies $0< C_1 \le |m_k| \le C_2 < \infty$. \end{lemma}
Note that by assumption  $\sum_{k=1}^{\infty} (1-|\lambda_k|^2 )<+\infty$  since $b^0\in \ell^2(\N)$. In particular $|\la_k| \to 1$.
\begin{lemma}\label{Carleson}
Let $D=\sum_j\lambda_jP_j$ be such that $|\la_k|<1,$ $\la_k \longrightarrow  1$  and let $b_k^0 = \sqrt{1-|\smash{\lambda_k}|^2}$. Then the following are equivalent:
\begin{enumerate}
\item[i)] 
$$\{b^0,Db^0,D^2 b^0,\dots\} \text{ is a frame for } \ell^2(N)$$
\item[ii)]
\begin{equation*}
\inf_{n} \prod_{k\neq n} \frac{|\la_n-\la_k|}{|1-\bar{\la}_n\la_k|}\geq \delta.
\end{equation*}
for some $\delta>0$.
\end{enumerate} 
\end{lemma}
In Lemma \ref {Carleson}, the assumption $\la_k \longrightarrow  1$ can be replaced by $\la_k \longrightarrow  -1$ and the lemma remains true. Its proof, below, is due to J. Antezana \cite{Ant14} and is a consequence of a theorem by Carleson \cite{KH88} about interpolating sequences in the Hardy space $\hsd$ of the unit disk in $\C$.

\begin {proof}[Proof of Lemma \ref {equiv}]
\ 

Let us first prove the sufficiency. Assume that $\{D^l b^0: l\in \N\}$ is a frame for $\ell^2(\N)$ with positive frame bounds $A$, $B$, and let $b \in \ell^2(\N)$ such that $b_k = m_k b^0_k$ with $0< C_1 \le |m_k| \le C_2 < \infty$. Let $x \in \ell^2(\N)$ be an arbitrary vector and define $y_k=\overline {m_k}x_k$. Then $y \in \ell^2(\N)$ and
$C_1 \|x\|_2 \le \|y\| \le C_2\|x\|_2$. Hence
\[
C_1^2A\|x\|^2_2\le \sum\limits_l |\langle y, D^lb^0\rangle |^2=\sum\limits_l\langle x, D^lb\rangle |^2\le C^2_2B\|x\|^2_2,
\]
and therefore $\{D^l b: l\in \N\}$ is a frame for $\ell^2(\N)$.

Conversely, let $b \in \ell^2(\N)$  and assume that $\{D^l b: l\in \N\}$ is a frame for $\ell^2(N)$  with frame bounds $A^\prime$ and $B^\prime$. 
Then for any vector $e_k$ of the standard orthonormal basis of $\ell^2(\N)$, we have 
\[
A^\prime  \le \sum_{l=0}^{\infty} |<e_k,D^lb>|^2=\frac{|b_k|^2}{1-|\lambda_k|^2}\le B^\prime.
\]
Thus $\sqrt{A^\prime} b^0_k\le |b_k|\le \sqrt{B^\prime} b^0_k$ for all $k$. Thus, the sequence $\{m_k\}\subset \C$ defined by $b_k=m_kb^0_k$ satisfies $\sqrt{A^\prime} \le |m_k|\le \sqrt{B^\prime}$. 

Let $x \in \ell^2(\N)$ be an arbitrary vector and define now $y_k=\frac 1 {\overline {m_k}}x_k$. Then $y \in \ell^2(\N)$ and
\[
\frac {A^\prime} {B^{\prime }}\|x\|^2_2\le \sum\limits_l |\langle x, D^lb^0\rangle |^2=\sum\limits_l |\langle y, D^lb\rangle |^2\le \frac {B^\prime} {A^{\prime }}\|x\|^2_2.
\]
and so $\{D^l b^0: l\in \N\}$ is a frame for $\ell^2(\N)$.
\end{proof}

The proof of  Lemma \ref {Carleson} relies on a Theorem  by Carleson on interpolating sequences in the Hardy space $\hsd$ on the open unit disk $\D$ in the complex plane. If $H(\mathbb{D})$ is the vector space of holomorphic functions on $\D$, $\hsd$ is defined as
$$
\hsd=\Big\{f \in H(\mathbb{D}): f(z)=\sum_{n=0}^\infty a_n z^n \ \ \mbox{for some sequence $\{a_n\}\in\ell^2(\N)$}\Big\}.
$$
Endowed with the inner product between $f=\sum_{n=0}^\infty a_n z^n$ and $g= \sum_{n=0}^\infty a^\prime_n z^n$ defined by $\langle f,g\rangle= \sum a_n\overline {a^\prime_n}$, $\hsd$ becomes a Hilbert space isometrically isomorphic to $\ell^2(\N)$ via the isomorphism $\Phi(f)=\{a_n\}$.
\begin {definition}
A sequence $\{\lambda_k\}$ in $\D$ is an interpolating sequence for $\hsd$ if for any function $f \in \hsd$ the sequence $\{ \frac {f(\lambda_k)}{\sqrt{1-|\lambda_k|^2}}\}$ belongs to $\ell^2(\N)$ and conversely, for any $\{c_k\} \in \ell^2(\N)$, there exists a function $f\in \hsd$ such that $f(\lambda_k)=c_k$.
\end{definition}
\begin{proof} [Proof of Lemma \ref {Carleson} ]
\ 

Let $\Tc_{k}$, denote the vector in $\ell^2(\N)$ defined by $\Tc_{k} = (1,\lambda_k,\lambda_k^2,\dots)$, and $x \in \ell^2(\N)$. Then 

$$
\sum_{l=0}^{\infty} |<x,D^l b^0>|^2= \sum_{l=0}^{\infty}\big{|} \sum_{k=1}^{\infty} x_k \lambda_k^l\sqrt{1-\smash{\lambda_k^2}}  
\big{|}^2 = \sum_{s=1}^{\infty}\sum_{t=1}^{\infty} \frac {<\Tc_s,\Tc_t>} {\|\Tc_s\|_2\|\Tc_t\|_2} x_s \overline{x_t}.
$$
Thus, for $\{D^lb^0: l=0,1,\dots\}$ to be a frame of $\ell^2(\N)$, it is necessary and sufficient that the Gramian $G_{\Lambda}=\{G_{\Lambda}(s,t)\} = \big{\{}\frac {<\Tc_s,\Tc_t>} {\|\Tc_s\|_2\|\Tc_t\|_2} \big{\}}$ be a bounded invertible operator on $\ell^2(N)$ (Note that $G_{\Lambda}$ is then the frame operator for $\{D^lb^0: l=0,1,\dots\}$). 

Equivalently, $\{D^lb^0: l=0,1,\dots\}$ is a frame of $\ell^2(\N)$ if and only if the sequence $\{\widetilde {\Tc_j}=\frac {\Tc_j} {\|\Tc_j\|_2}\}$ is a Riesz basic sequence in $\ell^2(\N)$, i.e., there exist constants $0<C_1\le C_2<\infty$ such that 
\[
C_1\|c\|^2_2\le\|\sum_jc_j\widetilde{\Tc}_j\|^2_2\le C_2\|c\|^2_2 \quad \text {for all } c \in \ell^2(\N).
\]
By the isometric map $\Phi$ from $\ell^2(\N)$ to $\hsd$ defined above, $\{D^lb^0: l=0,1,\dots\}$ is a frame of $\ell^2(\N)$ is a frame if and only if the sequence $\{\tilde k_{\lambda_j}=\Phi (\widetilde {\Tc_j}) $ is a Riesz basic sequence in $\hsd$. 

Let $k_{\lambda_j}=\Phi(\Tc_j)$. It is not difficult to check that for any $f\in \hsd$, $\langle f, k_{\lambda_j}\rangle=f(\lambda_j)$ and that $\{\lambda_j\}$ is an interpolating sequence in $\hsd$ if and only if  $G_{\Lambda}=\big(\langle \tilde k_{\lambda_j},\tilde k_{\lambda_j} \rangle\big)$ is a bounded invertible operator on $\ell^2(\N)$. By Carleson' s Theorem \cite {KH88}, this happens if and only if \eqref {carleson-cond} is satisfied.
\end {proof}
Frames of the form $\{D^lb_i: i \in \Omega, l=0\dots,l_i\}$ for the case when $|\Omega|\ge 1$ or when the projections $P_j$  have finite rank but possibly greater than or equal to $1$ can be  easily found by using Theorem \ref {OnePointFrame}. For example, if $|\Omega|=2$, $P_j(\ell^2(\N))$ has dimension $1$ for $j \in \N$, $b_1$, $\{\lambda_k\}$ satisfies the conditions of Theorem \ref {OnePointFrame} and $b_2$ is such that $b_2(k)=m_k\sqrt{1-|\lambda_k|^2}$ for some sequence $\{m_k\}$ satisfying $ |m_k| \le C< \infty$. To construct frames for the case when the projections $P_j$  have finite rank but possibly greater than or equal to $1$, we note that there exist orthogonal subspaces $W_1,\dots,W_N$ of $\ell^2(\N)$  such that  operator $D_i$ on each $W_i$ either has finite dimensional range, or satisfies  the condition of Theorem  \ref {OnePointFrame}.


\bibliographystyle{amsplain}
\providecommand{\bysame}{\leavevmode\hbox to3em{\hrulefill}\thinspace}
\providecommand{\MR}{\relax\ifhmode\unskip\space\fi MR }
\providecommand{\MRhref}[2]{%
  \href{http://www.ams.org/mathscinet-getitem?mr=#1}{#2}
}
\providecommand{\href}[2]{#2}

\end{document}